\documentclass[12pt,a4paper]{amsart}
\input xy
\xyoption{all}
\usepackage{amssymb}
\usepackage{amsmath}
\usepackage{amscd}
\usepackage{latexsym}
\begin{document}

\newtheorem{lemma}{Lemma}[section]
\newtheorem{prop}[lemma]{Proposition}
\newtheorem{cor}[lemma]{Corollary}
\newtheorem{thm}[lemma]{Theorem}
\newtheorem{con}[lemma]{Conjecture}

\theoremstyle{definition}
\newtheorem{rem}[lemma]{Remark}
\newtheorem{rems}[lemma]{Remarks}
\newtheorem{defi}[lemma]{Definition}
\newtheorem{ex}[lemma]{Example}
\newcommand{\C}{\mathbb C}
\newcommand{\R}{\mathbb R}
\newcommand{\Z}{\mathbb Z}
\newcommand{\MaxG}{C^{\ast}G}
\newcommand{\vNG}{{\mathcal N}\!(G)}
\newcommand{\vNH}{{\mathcal N}\!(H)}
\newcommand{\ldg}{\ell^2(G)}
\newcommand{\ldh}{\ell^2(H)}
\newcommand{\cg}{\C{G}}

\title[$L^2$-Index Theorem]{Atiyah's $\,L^2$-Index theorem}
\author{Indira Chatterji and Guido Mislin}
\address{Mathematics Department, Cornell University,
Ithaca NY 14853, USA\\
Mathematics Department, ETHZ, 8092 Z\"urich,
SWITZERLAND}\email{indira@math.cornell.edu, mislin@math.ethz.ch}
\date{July 2, 2003}
\maketitle
\section{Introduction}\label{Intro}
The $L^2$-Index Theorem of Atiyah \cite{atiyah} expresses the index of an elliptic operator on a closed manifold $M$ in terms of the $G$-equivariant index of some regular covering $\widetilde{M}$ of $M$, with $G$ the group of covering transformations. Atiyah's proof is analytic in nature. Our proof is algebraic and involves an embedding of a given group into an acyclic one, together with naturality properties of the indices.
\section{Review of the $L^2$-Index Theorem}\label{2.}
The main reference for this section is Atiyah's paper \cite{atiyah}. All manifolds considered are smooth Riemannian, without boundary. Covering spaces of manifolds carry the induced smooth and Riemannian structure. Let $M$ be a closed manifold and let $E$, $F$ denote two complex (Hermitian) vector bundles over $M$. Consider an elliptic pseudo-differential operator
$$ D: C^{\infty}(M, E)\to C^{\infty}(M, F)$$
acting on the smooth sections of the vector bundles. One defines its space of solutions 
$$S_D=\{s\in C^{\infty}(M, E)\,|\,Ds=0\}.$$
The complex vector space $S_D$ has finite dimension (see \cite{ST}), and so has $S_{D^*}$ the space of solutions of the adjoint $D^*$ of $D$ where
$$ D^*: C^{\infty}(M, F)\to C^{\infty}(M, E)$$
is the unique continuous linear map satisfying
$$\!\left<Ds,s'\right>=\!\int_M\!\!\left<Ds(m),s'(m)\right>_Fdm=\left<s,D^*s'\right>=\!\int_M\!\!\left<s(m),D^*s'(m)\right>_Edm$$
for all $s\in C^{\infty}(M, E)$, $s'\in C^{\infty}(M, F)$. One now defines the \emph{index} of $D$ as follows:
$$\operatorname{Index}(D)=\dim_\C (S_D) - \dim_\C (S_{D^*})\in\mathbb{Z}.$$
An explicit formula for $\operatorname{Index}(D)$ is given by the famous Atiyah-Singer Theorem (cf.~\cite{as}). Consider a not necessarily connected, regular covering $\pi:\widetilde{M}\to M$ with countable covering transformation group $G$. The projection $\pi$ can be used to define an elliptic operator
$$\widetilde{D}:=\pi^*(D):C^{\infty}_c(\widetilde{M},\pi^*E)\to C^{\infty}_c(\widetilde{M},\pi^*F).$$
Denote by $S_{\widetilde{D}}$ the closure of $\{s\in C^{\infty}_c(\widetilde{M},\pi^*E)|\widetilde{D}s=0\}$ in $L^2(\widetilde{M},\pi^*E)$. Let $\widetilde{D}^*$ denote the adjoint of $\widetilde{D}$. The space $S_{\widetilde{D}}$ is not necessarily finite dimensional, but being a closed $G$-invariant subspace of 
the $L^2$-completion $L^2(\widetilde{M},\pi^*E)$ of the space of smooth sections with compact supports $C^{\infty}_c(\widetilde{M},\pi^*E)$, its von Neumann dimension is therefore defined as follows. Write 
$$\vNG=\{P:\ldg\to\ldg\hbox{ bounded and }G\hbox{-invariant}\}$$
for the group von Neumann algebra of $G$, where $G$ acts on $\ldg$ via the right regular representation. Then $S_{\widetilde{D}}$ is a finitely generated Hilbert $G$-module and hence can be represented by an idempotent matrix $P=(p_{ij})\in M_n(\vNG)$ (recall that a finitely generated Hilbert $G$-module is isometrically $G$-isomorphic to a Hilbert $G$-subspace of the Hilbert space $\ldg^n$ for some $n\geq 1$, see \cite{Eckmann}). One then sets
$$\dim_{G}(S_{\widetilde{D}})=\sum_{i=1}^n\left<p_{ii}(e),e\right>=\kappa(P)\in \R,$$
where by abuse of notation $e$ denotes the element in $\ldg$ taking value 1 on the neutral element $e\in G$ and 0 elsewhere (see Eckmann's survey \cite{Eckmann} on $L^2$-cohomology for more on von Neumann dimensions). The map $\kappa:M_n(\vNG)\to\C$ is the Kaplansky trace. One defines the \emph{$L^2$-index} of $\widetilde{D}$ by
$$\operatorname{Index}_G(\widetilde{D})=\dim_{G}(S_{\widetilde{D}}) - \dim_{G}(S_{\widetilde{D}^*}).$$
We can now state Atiyah's $L^2$-Index Theorem.
\begin{thm}[Atiyah \cite{atiyah}]\label{principal}For $D$ an elliptic pseudo-differential operator on a closed Riemannian manifold $M$
$$\operatorname{Index}(D)=\operatorname{Index}_G({\widetilde{D}})$$
for any countable group $G$ and any lift ${\widetilde{D}}$ of $D$ to a regular $G$-cover $\widetilde{M}$ of $M$.
\end{thm}
In particular, the $L^2$-index of ${\widetilde{D}}$ is always an integer, even though it is a priori given in terms of real numbers.
The following serves as an illustration of the $L^2$-Index Theorem.
\begin{ex}[Atiyah's formula \cite{atiyah}]\label{deRham}
Let $\Omega^{\bullet}$ be the de Rham complex of complex valued differential forms on the closed connected manifold $M$ and consider the de Rham differential $D= d +d^*: \Omega^{ev}\to \Omega^{odd}$. Let $\pi:\widetilde{M}\rightarrow M$ be the universal cover of $M$ so that $G = \pi_1(M)$. Then
\begin{itemize}
\item $\operatorname{Index}(D)=\chi(M)$, the ordinary Euler characteristic of $M$.
\item $\operatorname{Index}_G({\widetilde{D}})=\sum_{j}(-1)^j\beta^j(M)$, the $L^2$-Euler characteristic of $M$.
\end{itemize}
The $\beta^j(M)$\rq s denote the $L^2$-Betti numbers of $M$. Thus the $L^2$-Index Theorem translates into Atiyah's formula
\[\chi(M)=\sum\limits_j(-1)^j\beta^j(M).\]
We recall that the $L^2$-Betti numbers $\beta^j(M)$ are in general not integers. For instance, if $\pi_1(M)$ is a finite group, one checks that
\[\beta^j(M)=\frac{1}{|\pi_1(M)|}b^j(\widetilde{M}),\]
where $b^j(\widetilde{M})$ stands for the ordinary $j$'th Betti number of the universal cover $\widetilde{M}$ of $M$. In particular, for $1<|\pi_1(M)|<\infty$, $\beta^0({M}) = 1/{|\pi_1(M)|}$ is not an integer and the $L^2$-Index Theorem reduces to the well-known fact that
\[ \chi(M) = \frac{\chi(\widetilde{M})}{|\pi_1(M)|}.\]
It is a conjecture (Atiyah Conjecture) that for a general closed connected manifold $M$ the $L^2$-Betti numbers $\beta^j(M)$ are always rational numbers, and even integers in case that $\pi_1(M)$ is torsion-free. For some interesting examples, which might lead to counterexamples, see Dicks and Schick \cite{DS}.\end{ex}
\section{Hilbert modules}
Recall that for $H<G$ and $X$ an $H$-space, the \emph{induced} $G$-space is
$$G\times_HX=(G\times X)/H$$
where $H$ acts on $G\times X$ via $h\cdot(g,x)=(gh^{-1},hx)$ and the left $G$-action on $G\times_HX$ is given by $g\cdot[k,x]=[gk,x]$ (where $[k,x]$ denotes the class of the pair $(k,x)\in G\times X$ in $G\times_HX$). For $A\subseteq\ldh^n$ a Hilbert $H$-module one defines ${\rm Ind}_H^G(A)$ the \emph{induced} Hilbert $G$-module as follows:
$${\rm Ind}_H^G(A)=\{f:G\to A, f(gh)=h^{-1}f(g), \sum_{\gamma\in G/H}\|f(\gamma)\|^2<\infty\}.$$
On ${\rm Ind}_H^G(A)$ the action of $G$ is given as follows:
$$(\gamma\cdot f)(\mu)=f(\gamma^{-1}\mu),\eqno{\gamma,\mu\in G\hbox{ and }f\in{\rm Ind}_H^G(A).}$$
For $\widetilde{M}$ an $H$-free, cocompact Rieman\-nian manifold and $\widetilde{D}$ an $H$-equi\-variant pseudo-differential operator on $\widetilde{M}$, one can express the lift $\overline{D}$ of $\widetilde{D}$ to $\overline{M}=G\times_H \widetilde{M}$ as follows. Fix a set $R$ of representatives for $G/H$ and write $\pi:\overline{M}\to\widetilde{M}$ for the projection; a section $\overline{s}\in C^{\infty}_c(\overline{M},\pi^*E)$ is a collection
$$\overline{s}=\{\widetilde{s}_r\}_{r\in R},$$
where $\widetilde{s}_r\in C^{\infty}_c(\widetilde{M},E)$ is the zero section for all but finitely many $r$'s, and $\overline{s}([g,\widetilde{m}])=\widetilde{s}_r(h\widetilde{m})$, if  $[r,h\widetilde{m}]=[g,\widetilde{m}]\in G\times_H \widetilde{M}$. Now the lift $\overline{D}$ of $\widetilde{D}$ to $\overline{M}=G\times_H \widetilde{M}$ satisfies
$$\overline{D}\overline{s}=\{\widetilde{D}\widetilde{s}_r\}_{r\in R}.$$
\begin{lemma}\label{mi-trivial}Let $M$ be a closed Riemannian manifold, $D$ a pseudo-differential operator on $M$ and $\widetilde{M}$ a regular cover of $M$ with countable transformation group $H$. Consider an inclusion $H<G$ and form the regular cover $\overline{M}=G\times_H \widetilde{M}$ of $M$. Then for the lifts $\widetilde{D}$ of $D$ to $\widetilde{M}$ and ${\overline{D}}$ of $\widetilde{D}$ to $\overline{M}$,
$$\operatorname{Index}_H({\widetilde{D}})=\operatorname{Index}_G({\overline{D}}).$$\end{lemma}
\begin{proof}
It is enough to see that $S_{\overline{D}}\cong{\rm Ind}_H^G(S_{\widetilde{D}})$. Indeed, it is well-known (see \cite{Eckmann}) that for a Hilbert $H$-module $A$ one has
$$\dim_H(A)=\dim_G({\rm Ind}_H^G(A)).$$
For $R$ a fixed set of representatives for $G/H$, the map
\begin{eqnarray*}
\varphi_R:{\rm Ind}_H^G(S_{\widetilde{D}})&\to& S_{\overline{D}}\\
f&\mapsto&\{f(r)\}_{r\in R}
\end{eqnarray*}
is well-defined by $H$-equivariance of the elements of $S_{\widetilde{D}}$ and one checks that it defines a $G$-equivariant isometric bijection. Similarly for the adjoint operators.
\end{proof}
The following example is a particular case of the previous lemma.
\begin{ex}\label{trivial}Let us look at the case $\widetilde{M}=M\times G$. A section $\widetilde{s}\in C^{\infty}_c(\widetilde{M},\pi^*E)$ is an element $\widetilde{s}=\{s_g\}_{g\in G}$ where $s_g\in C^{\infty}(M, E)$ and $s_g=0$ for all but finitely many $g$'s. Note that $L^2(\widetilde{M},\pi^*E)$ can be identified with $\ldg\otimes L^2(M,E)$. Now
$$\widetilde{D}\widetilde{s}=\{Ds_g\}_{g\in G}\in C^{\infty}_c(\widetilde{M},\pi^*F)$$
and hence $S_{\widetilde{D}}$ may be identified with $\ldg\otimes S_D\cong\ldg^d$, where $d=\dim_{\C}(S_D)$. In this identification the projection $P$ onto $S_{\widetilde{D}}$ becomes the identity in $M_d(\vNG)$ and thus
$$\dim_G(S_{\widetilde{D}})=\sum_{i=1}^d\left<e,e\right>=d=\dim_{\C}(S_D).$$
A similar argument for $D^*$ shows that in this case not only the $L^2$-Index of $\widetilde{D}$ coincides with the Index of $D$, but also the individual terms of the difference correspond to each other. This is not the case in general, see Example \ref{deRham}.
\end{ex}
\section{On $K$-homology}\label{K}
Many ideas of this section go back to the seminal article by Baum and Connes \cite{BC}, which has been circulating for many years and has only recently been published.

\smallbreak

An elliptic pseudo-differential operator $D$ on 
the closed manifold $M$ can also be used to define an element $[D]\in K_0(M)$, the $K$-homology of $M$, 
and according to Baum and Douglas \cite{bd}, all elements of $K_0(M)$
are of the form $[D]$.
The
index defined in Section \ref{2.} extends to a well-defined homomorphism (cf.~\cite{bd})
$$\operatorname{Index}: K_0(M)\to \Z,$$
such that $\operatorname{Index}([D])=\operatorname{Index}(D)$. On the other hand, the projection $\operatorname{pr}:M\to \{pt\}$ induces, after identifying $K_0(\{pt\})$ with $\Z$, a homomorphism
$$\operatorname{pr}_* : K_0(M)\to\Z,\eqno{(*)}$$
which, as explained in \cite{bd}, satisfies
$$\operatorname{pr}_* ([D]) = \operatorname{Index}([D]).$$
More generally (cf.~\cite{bd}), for a not necessarily finite CW-complex $X$, every $x\in K_0(X)$ is of the form $f_*[D]$ for some $f:M\to X$, and $K_0(X)$ is obtained as a colimit over $K_0(M_{\alpha})$, where the $M_{\alpha}$ form a directed system consisting of closed Riemannian manifolds (these homology groups $K_0(X)$ are naturally isomorphic to the ones defined using the Bott spectrum; sometimes, they are referred to
as $K$-homology groups with {\em compact supports}). 
The index map from above extends to a homomorphism
$$\operatorname{Index}: K_0(X)\to \Z,$$
such that $\operatorname{Index}(x)=\operatorname{Index}([D])$ if
$x=f_*[D]$, with $f:M\to X$. 

We now consider the case of $X=BG$, the
classifying space of the discrete group $G$, and obtain thus
for any $f:M\to BG$ a commutative diagram 
$$\begin{CD}
K_0(M)  @>{\operatorname{Index}}>> \mathbb{Z}\\
@V{f_*}VV      @|\\
K_0(BG)    @>{\operatorname{Index}}>> \ \mathbb{Z}.\end{CD}$$
Note that $(*)$ from above implies the following naturality
property for the index homomorphism.
\begin{lemma}\label{IndiceNaturel}For any homomorphism $\varphi:H\to G$ one has a commutative diagram
$$\begin{CD}
K_0(BH)  @>{\operatorname{Index}}>> \mathbb{Z}\\
@V{(B\varphi)_*}VV      @|\\
K_0(BG)    @>{\operatorname{Index}}>> \ \mathbb{Z}.\end{CD}$$
\end{lemma}
\begin{flushright}$\square$\end{flushright}
We now turn to the $L^2$-index
of Section \ref{2.}. It extends to a homomorphism
$$\operatorname{Index}_G:K_0(BG)\to\mathbb{R}$$
as follows. Each $x\in K_0(BG)$ is of the form $f_*(y)$ for some $y=[D]\in K_0(M)$, $f:M\to BG$, $M$ a closed smooth manifold and $D$ an elliptic operator on $M$.
Let $\widetilde{D}$ be the lifted operator to $\widetilde{M}$, 
the $G$-covering space induced by $f:M\to BG$. Then put
$$\operatorname{Index}_G(x):=\operatorname{Index}_G(\widetilde{D}).$$
One checks that $\operatorname{Index}_G(x)$ is indeed
well-defined, either
by direct computation, or by identifying it with $\tau(x)$, where
$\tau$ denotes the composite of the assembly map $K_0(BG)\to
K_0(C^*_rG)$ with the natural trace $K_0(C^*_rG) \to \R$ (for this
latter point of view, see
Higson-Roe \cite{hr}; for a discussion of the assembly map see e.g. Kasparov~\cite{Kas}, or Valette~\cite{Valette}). The following naturality property of this index map is a consequence of Lemma \ref{mi-trivial}.
\begin{lemma}\label{GIndiceNaturel}
For $H<G$ the following diagram commutes
$$\begin{CD}
K_0(BH)  @>{\operatorname{Index}_H}>> \mathbb{R}\\
@VVV      @|\\
K_0(BG)    @>{\operatorname{Index}_G}>> \ \mathbb{R}.\end{CD}$$
\end{lemma}
\begin{flushright}$\square$\end{flushright}
Atiyah's $L^2$-Index Theorem \ref{principal} for a given $G$ can now be expressed as the statement (as already observed in \cite{hr})
$$\operatorname{Index}_G=\operatorname{Index}:K_0(BG)\to\mathbb{R}.$$
\section{Algebraic proof of Atiyah's $L^2$-index theorem}\label{proof}
Recall that a group $A$ is said to be \emph{acyclic} if $H_*(BA,\Z)=0$ for $*>0$. For $G$ a countable group, there exists an embedding $G \rightarrow A_G$ into a countable acyclic group $A_G$. There are many constructions of such a group $A_G$ available in the literature, see for instance Kan-Thurston \cite[Proposition 3.5]{kt}, Berrick-Varadarajan \cite{BerrVarad} or Berrick-Chatterji-Mislin
\cite{bcm}; these different constructions are to be compared in Berrick's forthcoming work \cite{Berrick}. It follows that the suspension $\Sigma BA_G$ is contractible, and therefore the inclusion $\{e\} \rightarrow A_G$ induces an isomorphism
\[K_0(B\{e\})\overset{\cong}{\longrightarrow} K_0(BA_G).\]
Our strategy is as follows. We show that the Atiyah $L^2$-Index Theorem holds in the special case of acyclic groups, and finish the proof combining the above embedding of a group into an acyclic group.
\begin{proof}[Proof of Theorem \ref{principal}] If a group $A$ is acyclic, the equation $\operatorname{Index}_A=\operatorname{Index}$ follows from the diagram
$$\begin{CD}
K_0(BA)@>{\operatorname{Index}_A}>> \mathbb{R}@<{\operatorname{Index}}<< K_0(BA)\\
@A{\cong}AA     @AAA   @A{\cong}AA \\
K_0(B\{e\})@>{\operatorname{Index}_{\{e\}}}>{\cong}>  \mathbb{Z}@<{\operatorname{Index}}<{\cong}< K_0(B\{e\})\\
\end{CD}$$
because $\operatorname{Index}_{\{e\}}=\operatorname{Index}$ on the bottom line. For a general group $G$, consider an embedding into an acyclic group $A_G$ and complete the proof by using Lemma \ref{mi-trivial}, together with Lemmas \ref{IndiceNaturel} and \ref{GIndiceNaturel}.
\end{proof}


\end{document}